\documentclass[12pt,a4paper]{amsart}
  \pdfoutput=1
  
  \usepackage{amsmath,amssymb,amsthm}
  \usepackage{bm}  
  \usepackage{mathtools}
 \usepackage{tikz}
  \usepackage{thmtools}
  \usepackage{float} 
  
  \usepackage{xcolor} 	
  \usepackage{hyperref}
  \hypersetup{
  	colorlinks,
      linkcolor={red!60!black},
      citecolor={green!60!black},
      urlcolor={blue!60!black},
  }
  \usepackage[maxbibnames=99]{biblatex}
  \usepackage[utf8]{inputenc}
  \usepackage[T1]{fontenc}
  \usepackage{lmodern}
  \usepackage[babel]{microtype}
  \usepackage[english]{babel}
  
  \usepackage{geometry}
  \usepackage{enumitem}

 \theoremstyle{plain}
 \newtheorem{thm}{Theorem}[section]
 
 \newtheorem{prop}[thm]{Proposition}
 
 \newtheorem{claim}[thm]{Claim}
 \newtheorem{cor}[thm]{Corollary}
 \newtheorem{lem}[thm]{Lemma}
 
 \newtheorem{obs}[thm]{Observation}

 \newtheorem{conj}[thm]{Conjecture}

 \newtheorem*{thm*}{Theorem}
 \newtheorem*{cor*}{Corollary}
 
 \theoremstyle{definition}
 \newtheorem{rem}[thm]{Remark}

 \newtheorem{cond}[thm]{Condition}
 
 \newenvironment{nbiz}{\par\noindent{\itshape Proof:}\ }{}

\title{The Lovász-Cherkassky theorem in countable graphs}

\author{Attila Jo\'{o}}
\thanks{The author would like to thank the generous support of the Alexander 
von Humboldt Foundation and NKFIH 
OTKA-129211}
\address{Attila Jo\'{o},
Department of Mathematics, University of Hamburg, Bundesstra{\ss}e 55 (Geomatikum), 20146 Hamburg, Germany}
\email{attila.joo@uni-hamburg.de}
\address{Attila Jo\'{o},
Logic, Set theory and topology department, Alfr\'{e}d R\'{e}nyi Institute of Mathematics,  13-15 Re\'{a}ltanoda St., 
Budapest, Hungary}
\email{jooattila@renyi.hu}


\begin{document}

\begin{abstract}
Lovász and Cherkassky discovered in the 1970s independently that if $ G $ is a finite graph with a given set $ T $ of terminal 
vertices such that $ G $ is inner Eulerian, then the maximal number of edge-disjoint paths  connecting distinct vertices in $ T 
$  is $ \sum_{t\in T}\lambda(t, T-t)  $ where $\lambda $ is the local edge-connectivity function. The optimality of a system of 
edge-disjoint $ T $-paths in the Lovász-Cherkassky theorem is witnessed by the existence of certain cuts  by Menger's 
theorem. The infinite generalisation of Menger's theorem by Aharoni and Berger (earlier known as the Erdős-Menger 
Conjecture) together with the characterization of infinite Eulerian graphs due to Nash-Williams makes it possible  to 
generalise the theorem for infinite graphs in a structural way. The aim of this paper is to formulate this generalisation  and 
prove it for countable graphs.
\end{abstract}

\maketitle

\section{Introduction}
There are several deep results and conjectures in infinite combinatorics whose restriction to finite structures is a well-known 
classical theorem.  For example \cite{aharoni1983duality, aharoni1984konig} by Aharoni is known as Hall's 
and König's theorem when only finite graphs are considered and it is based on  the results \cite{aharoni1983general, 
aharoni1984another, aharoni1984marriage} by Aharoni, Nash-Williams and Shelah. The finite case of the Aharoni-Berger 
theorem \cite{aharoni2009menger}  (earlier known as the Erdős-Menger Conjecture)  is known as Menger's 
theorem and the Matroid Intersection Conjecture \cite{aharoni1998intersection} by Nash-Williams (which is only settled in the 
countable case \cite{joo2020MIC})  extends the Matroid Intersection Theorem 
\cite{edmonds2003submodular} of Edmonds. 

There are several common aspects of the
problems above. For example, assuming the finiteness of the involved structures simplifies the proof significantly. Indeed, the 
deletion of a cleverly chosen edge  gives rise to an inductive argument as well as the 
application of an ``augmenting path''. In contrast to the finite case, the deletion of a single element of an infinite set does not 
decrease its size, furthermore,  an infinite sequence of iterative augmentations may fail to give a well-defined
``limit object''. Another similarity between these statements is that they express a certain   
``complementary slackness'' condition  between suitable primal and dual objects: 
a matching $ M $ in $ G=(A,B,E) $ and a vertex-cover $ C $ consisting of a single vertex from each $ e\in M $;  a disjoint 
path-system $ \mathcal{P} 
$ between $ A $ and $ B $ in $ G=(V,E) $ with $ A,B\subseteq V $ and an $ AB $-separation $ S\subseteq V $ consisting of 
 a choice of a
single vertex from 
each $ P\in \mathcal{P} $;  a common 
independent  set $ 
I $ of matroids $ M_0 $ and $ M_1 $ and a bipartition $ E=E_0\sqcup E_1 $ of their common edge set such that $ E_i\cap I $ 
spans $ 
E_i $ in $ M_i $ for $ i\in \{ 0,1 \} $. Alternative characterizations of ``primal optimality'' can be given through the concept of 
strong maximality. Let 
us call an element $ X$ of  a set family $ \mathcal{X} $ strongly maximal in $ \mathcal{X} $ if $ \left|Y\setminus X\right| 
\leq\left|X\setminus 
Y\right| $ for every $ Y\in \mathcal{X} $. Note that if $ \mathcal{X} $ has only finite elements, then `strongly maximal' means 
 `maximum size', however, in general having maximum size is a much weaker property than strong maximality. It is 
 known  in the three problems we mentioned  that  the strong 
 maximality of a 
 matching/disjoint path 
system/common independent set is equivalent with the existence of a vertex-cover/separation/bipartition  such 
that the corresponding complementary slackness conditions are satisfied.

The starting point of our investigation is the following result obtained by Lovász and 
Cherkassky independently in the 1970s:

\begin{thm}[Lovász-Cherkassky theorem, \cite{lovasz1976some, cherkasskiy1977}]\label{t:LCh}
Let $ G $ be a finite graph and let $ T\subseteq V(G) $ such that $ G $ is inner Eulerian (i.e. $ d_G(v) $ is even for every $ v\in 
V(G)\setminus T $).  Then the 
maximal number of pairwise edge-disjoint $ T $-paths\footnote{A $ T $-path is a path connecting distinct vertices in $ T $ 
without having 
internal vertex in $ T $.}  is

\[ \frac{1}{2}\sum_{t\in T}\lambda_G(t, T-t),  \]

where $ \lambda_G(t, T-t) $ stands for the maximal number of pairwise edge-disjoint paths between $ t $ and $ T-t $.
\end{thm}

 The literal extension of Theorem \ref{t:LCh} to infinite graphs fails. Indeed, let $ G=(V,E) $ be the graph we obtain from  the 
 star $ K_{1,3} $ by attaching a one-way 
infinite path to its central vertex (see Figure \ref{fig: cross}). We define $ T $ to be the set of vertices of degree one. Then we 
have only 
even degrees in $ V\setminus 
T $ 
and the maximal number of edge-disjoint $ T $-paths is $ 1 $ although  $ \frac{1}{2}\sum_{t\in T}\lambda(t, T-t)=\frac{3}{2} $.

\begin{figure}[h]
\centering
\begin{tikzpicture}

\node[circle, fill=black] (v1) at (0,0) {};
\node[circle, draw] (v2) at (0,-1) {};
\node[circle, fill=black] (v3) at (-1,-1) {};
\node[circle, fill=black] (v4) at (0,-2) {};
\node[circle, draw] (v5) at (1,-1) {};
\node[circle, draw] (v6) at (2,-1) {};
\node[circle, draw] (v7) at (3,-1) {};
\node at (3.5,-1) {$\dots$};

\draw  (v1) edge (v2);
\draw  (v2) edge (v3);
\draw  (v2) edge (v4);
\draw  (v2) edge (v5);
\draw  (v5) edge (v6);
\draw  (v6) edge (v7);
\end{tikzpicture}
\caption{The failure of the literal infinite generalisation of the Lovász-Cherkassky theorem. Elements of $ T $ are 
black.}\label{fig: 
failure}
\end{figure}
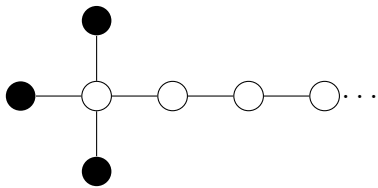 \label{fig: cross}

The reason of this discrepancy is that after allowing $ G $ to be infinite  the condition ``$ G $ is  Eulerian'' (i.e. $ E(G) 
$ can be 
partitioned into edge-disjoint cycles)  is no longer
 equivalent with the property that $ G $ has only even degrees.  Indeed, in the two-way infinite 
 path each 
 degree is $ 2 $ but it is 
 obviously not 
 Eulerian. On the other hand, graphs 
 with  infinite degrees can be easily Eulerian.
The characterization of infinite Eulerian graphs due to Nash-Williams is one of the fundamental theorems in infinite graph theory:
\begin{thm}[{Nash-Williams, \cite[p. 235, Theorem 3]{nash1960decomposition}}]\label{t: NW}
A (possibly infinite) graph is Eulerian if and only if it does not contain  an odd cut.\footnote{Infinite cardinals considered neither 
odd nor even.} 
\end{thm}

Simpler proofs for Theorem \ref{t: NW} were given by L. Soukup 
(\cite[Theorem 5.1 ]{soukup2011elementary})  and 
Thomassen  \cite{thomassen2017nash}
while its analogue  for 
directed graphs  (conjectured by Thomassen)  was settled affirmatively in \cite{joo2021dicycleDecomp}. Theorem 
\ref{t: NW} indicates that the condition ``for every $ v\in V\setminus T $: $ d(v) $ is even''  should be replaced by ``for every $ 
X\subseteq V\setminus T $: $ d(X) $ is 
not odd''  in order to allow infinite graphs. Note that in finite graphs the former condition is equivalent to the statement that 
`contracting $ T $ results
in an Eulerian graph' and by Theorem \ref{t: NW} the latter condition is equivalent to the same but  for graphs of any size.

The literal adaptation of the formula $ \frac{1}{2}\sum_{t\in T}\lambda(t, T-t) $ is also not really fruitful in the presence of 
infinite quantities. Consider for example  the 
graph $ (\{ u,v \}, E) $ with $ T= \{ u,v \}$  where $ E $ consists of $ \aleph_0 $  parallel edges between $ u $ and $ v $. 
Then any infinite $ \mathcal{P}\subseteq E $,  considered as a set of paths of length one, has the same size $ \aleph_0 $. 
It demonstrates that cardinality is an overly 
rough measure in the presence of 
infinite quantities and  urges us  to focus on combinatorial instead of quantitative properties of an optimal 
path-system 
in Theorem \ref{t:LCh}. In a finite graph a system $\mathcal{P}$ of edge-disjoint $T$-paths has  $ {\frac{1}{2}\sum_{t\in 
T}\lambda(t,T-t)} $ elements  if and only if $\mathcal{P}$ contains  
$\lambda(t,T-t)$ paths between $t$ and $T-t$ for each  $t\in T$. By Menger's theorem it is equivalent to the statement that for 
every 
$t\in T$ one can choose exactly one edge from each $P\in \mathcal{P}$ having $t$ as an end-vertex such that the resulting edge 
set $C$ is a cut separating $t$ from $T-t$. Now we are ready to state our main results: 

\begin{thm}\label{t:LCh inf}
Let $G$ be a  graph and let $T\subseteq V(G)$  be countable such that  there is no $ X\subseteq V(G)\setminus 
T $ 
where $ d_G(X) $ is an odd natural number. Then there exists a system 
$\mathcal{P}$ of edge-disjoint $T$-paths such that for every $t\in T$: one can choose exactly one edge from each $P\in 
\mathcal{P}$ having $t$ as an end-vertex in such a way that the resulting edge set $C$ is a cut separating $t$ and $T-t$. 
\end{thm}

We also prove the following closely related theorem.

\begin{thm}\label{t: LCh1 infIntro}
Let $G$ be a  graph and let $T\subseteq V(G)$  be countable such that  there is no $ X\subseteq V(G)\setminus 
T $ 
where $ d_G(X) $ is an odd natural number. Assume that for each $t\in T$  
there is a system $\mathcal{P}_t$ of edge-disjoint $T$-paths covering all the edges incident with $ t $. Then there exists a system 
$\mathcal{P}$ of 
edge-disjoint $T$-paths covering all the edges incident with any $ t\in T $. 
\end{thm}

We strongly believe that the countability of  $ T $  can be omitted in the theorems above.  However, based on the experience with 
the similar problems mentioned earlier, we suspect that the proof is significantly harder.
   
Mader gave in \cite{mader1978maximalzahl} a minimax theorem about the maximal number of edge-disjoint $ T $-paths   in 
arbitrary (i.e. not necessarily inner Eulerian) finite graphs. It can be considered as a generalisation of Theorem \ref{t:LCh}.  The 
structural and 
algorithmic aspects of the problem have been  a subject of interest ever since (see for example \cite{sebHo2004path}, 
\cite{keijsper2006linear}, \cite{babenko2017faster} and \cite{iwata2020blossom}) as well the analogous theorems 
considering 
vertex-disjoint \cite{gallai1964maximum} and internally vertex-disjoint \cite{mader1978maximalzahl} paths.

\begin{conj}\label{conj: strongly maxim}
Let $G$ be a  graph and let $T\subseteq V(G)$. Then there exists a strongly maximal system $ \mathcal{P} $ of 
edge-disjoint/vertex-disjoint/internally vertex-disjoint $ T $-paths in $ G $. 
\end{conj}

We  conjecture that the path-systems $ \mathcal{P} $ in Conjecture \ref{conj: strongly maxim} can be characterized in the 
way that it extends the corresponding minimax theorem to infinite graphs based on complementary slackness conditions. We 
discuss the details in Section \ref{sec: oultook}. Before we turn to the proof of our main results in Section \ref{sec: 
proofLovCHer}, we need to introduce some notation 
and recall a few results we are going to use in the proof. These are done in Sections \ref{sec: notation} and \ref{sec: premi} 
respectively.

\section{Notation}\label{sec: notation}
In graphs we allow parallel edges but not loops. Technically we represent a graph as a triple $ G=(V,E,I) $ where the 
\emph{incidence 
function} $ \boldsymbol{I}:E\rightarrow [V]^{2} $ determines the end-vertices of the edges. 
For $ X\subseteq V $ let $ \boldsymbol{\delta_G(X)}:=\{ e\in E:  \left|I(e)\cap X\right|=1  \} $ and we write $ 
\boldsymbol{d_G(X)} $ for $ 
\left|\delta_G(X)\right| $.  If  a graph $ G $ is obvious from the context, then we omit the subscript, furthermore, for a singleton $ 
\{ 
x \} $ we write simply $ \delta(x) $ and $ d(x) $. Instead of a graph $ G $ we may give a path-system or an edge set of a 
fixed graph $ G $ as a subscript of $ \delta $. In these cases we mean the corresponding subgraph of $ G $.  All the paths in the 
paper are finite. We refer 
sometimes the first vertex or last 
edge of a path. The context will always indicate according which direction we mean this. The first and last edges are the 
\emph{extremal edges} of the path. An $ AB $\emph{-path} for $ 
A,B\subseteq V $ 
is a path with first 
vertex in $ A $ 
last vertex in $ B $ and no internal vertices in $ A\cup B $.   A set  $ C\subseteq E $ is a \emph{cut} if $ C=\delta(X) $ for some 
$ 
X\subseteq 
V $. If 
$ G $ is connected then $ X $ is  determined by $ C $  up to taking complement and the $ v $\emph{-side}
of the cut $ C $ is the unique $ X $ with $ C=\delta(X) $ and $ v\in X $. We call $ \delta(X) $ an $ AB $-cut if $ A\subseteq X $ 
and $ B\cap X =\varnothing $ or the other way around. If $ A $ and $ B $ are singletons, $ A=\{ s \} $ and $ B=\{ t \} $ say, then 
we write simply $ st $-cut instead of $ \{ s \}\{ t \} $-cut.   In a connected graph 
$ G $,  a cut  $ \delta(X) 
$ is $ \subseteq 
$-minimal  if 
and only if the induced subgraphs $ G[X] $ and $ G[V\setminus X] $ are connected. We extend the definitions above for 
disconnected graphs $ G $ and cuts $ C $ living in a single connected component $ M $ by considering $ C $ as a cut in $ G[M] $.
For a $ U\subseteq V $ and a family $ \mathcal{F}=\{ X_u: u\in U \} $ of pairwise disjoint 
subsets of $ V $ with $ X_u\cap U=\{ u \} $, we define the graph  $ \boldsymbol{G/\mathcal{F}} $  obtained 
from $ G $ by \emph{contracting} $ X_u$ to $u $ for $ u\in U $ and deleting the resulting loops. More formally  
$V(G/\mathcal{F}):= 
(V\setminus \bigcup \mathcal{F} 
)\cup U,\ E(G/\mathcal{F}):= E\setminus \{ e\in E: (\exists u\in U) I(e)\subseteq X_u \} $ and  $ 
I(G/\mathcal{F})(e):=\{ i_\mathcal{F}(u), i_\mathcal{F}(v) \} $  where  $ I(e)=\{ u,v \} $ and

\[ i_\mathcal{F}(v) =\begin{cases} v &\mbox{if } v\notin \bigcup \mathcal{F} \\
u & \mbox{if } u\in X_u .  
\end{cases} \]

\section{Preliminaries}\label{sec: premi}
Menger's theorem and the other connectivity-related results that we recall in this section  have  four versions depending on 
if the graph is directed and if we consider vertex-disjoint or edge-disjoint paths. In all of these theorems  the two 
directed variants are 
equivalent as well as the two undirected variants which can be shown by simple techniques like splitting edges by a new vertex 
and blowing up vertices to a highly connected vertex sets. Furthermore, through replacing undirected edges by  back and forth 
directed ones the undirected vertex-disjoint version can be reduced to the directed one.
 
In this paper we deal only with undirected graphs and edge-disjoint paths so let us always formulate  immediately that variant 
even if historically other version was proved first.

Let a connected graph $ G=(V,E)$ and distinct $ s,t\in V$ be fixed.  For   $ 
\subseteq $-minimal $ st $-cuts $ C $ and $ D $ 
we write $ 
C\boldsymbol{\preceq} D $ if the $ 
s $-side of the cut $ C $ is a subset of the $ s $-side of $ D $. Note that the $ \subseteq $-minimal $ st $-cuts with $ \preceq $  
form a complete lattice.  
For a 
finite $ G $  the optimal (minimal-sized)
$ st $-cuts  form a distributive sublattice (see \cite{escalante1974note}) of it. In general graphs 
 the size of the cut is an overly rough measure for optimality. A structural infinite generalisation of the class of 
``optimal'' $ st $-cuts is provided by the Aharoni-Berger theorem:
  
\begin{thm}[Aharoni and Berger, \cite{aharoni2009menger}]\label{t: Inf Menger}
Let $ G$ be a (possibly infinite)  graph and let $ s,t\in V(G) $ be distinct. Then there is a system $ 
\mathcal{P} $ of 
edge-disjoint $ st $-paths and an $ st $-cut  $ C $ which is \emph{orthogonal} to $ \mathcal{P} $, i.e. $ C $ consists of 
choosing 
exactly one edge from each path in 
$ \mathcal{P} $.
\end{thm}

We say that the $ st $-cut $ C $  in Theorem \ref{t: Inf Menger} is an \emph{Erdős-Menger} $ st $\emph{-cut} and we let $ 
\boldsymbol{\mathfrak{C}(s,t)} $ be the set of such cuts.

\begin{thm}[J. \cite{joo2019complete}]\label{l: EMlattice}
$ (\mathfrak{C}(s,t),\preceq) $ is a complete lattice, although usually not a sublattice of all the minimal $ st $-cuts.
\end{thm}

Finally we introduce two more classes $ \mathfrak{C}^{-}(s,t) $ and $ \mathfrak{C}^{+}(s,t) $ of minimal $ st $-cuts with  
$  \mathfrak{C}^{-}(s,t)\cap \mathfrak{C}^{+}(s,t)=\mathfrak{C}(s,t) $ and
$ \mathfrak{C}^{+}(s,t):=\mathfrak{C}^{-}(t,s) $. Let  $ \boldsymbol{\mathfrak{C}^{-}(s,t) }$ consist of those minimal $ st 
$-cuts  $ C $ for which there  is a system $\mathcal{W}$ of  pairwise edge-disjoint 
paths starting at $ s $ and having $ C $ as the set of last edges  (considering the paths directed 
away from $s$). Such a $ 
\mathcal{W} $ is called an $ st $-\emph{wave} and plays an important role in the proof of Theorem \ref{t: Inf Menger}. The 
cut defined as the last edges of a paths-system $ \mathcal{W} $ is denoted by $ \boldsymbol{C_\mathcal{W}} $. If  $ \delta(s) $ 
is   
a wave (considering the edges as paths of length one), than we call it the \emph{trivial} $ st 
$\emph{-wave}.

\begin{lem}[{\cite[Lemma 3.8]{joo2019complete}}]\label{l: large wave}
 $ (\mathfrak{C}^{-}(s,t), \preceq) $  is a complete lattice and a sup-sublattice of all the minimal $ st $-cuts.  After the 
contraction of the 
$ s $-side of its largest element to $ s $, there is exactly one wave in the resulting system, namely the trivial one.
\end{lem}

We call an $ st $-wave $ \mathcal{W} $ \emph{large} if $ C_{\mathcal{W}} $ is the largest element of $ \mathfrak{C}^{-}(s,t) 
$. Note that if there is no non-trivial  $ st $-wave, then $ \delta(s) $ must be an Erdős-Menger $ 
st $-cut because $\mathfrak{C}(s,t)\subseteq \mathfrak{C}^{-}(s,t)\subseteq \{ \delta(s) \} $ and the left side is nonempty by 
Theorem \cite{aharoni2009menger}. This leads to the following conclusion:

\begin{cor}\label{cor: delta(s) coverable}
If there is no non-trivial $ st $-wave, then there is a system $ \mathcal{P} $ of edge-disjoint $ st $-paths covering $ \delta(s) $ 
and hence $ \mathfrak{C}=\{ \delta(s) \} $.
\end{cor}

\begin{thm}[Diestel and Thomassen, \cite{diestel2006cantor}]\label{t: Pym}
Assume that $ G $ is a (possibly infinite)  graph, $ s,t\in V(G) $ are distinct, furthermore, $ \mathcal{P} $ and  $ \mathcal{Q} $ 
are 
systems of edge-disjoint $ st $-paths. Then there exists a system $ \mathcal{R} $ of edge-disjoint $ st $-paths such that 
$ \delta_\mathcal{R}(s) \supseteq  \delta_\mathcal{P}(s) $ and $ \delta_\mathcal{R}(t) \supseteq  \delta_\mathcal{Q}(t) $.
\end{thm}

Let $ \mathcal{P} $ be a system of edge-disjoint $ st $-paths and let $ \mathcal{W} $ be  a large $ st $-wave. By 
contracting the $ t $-side of $ C_{\mathcal{W}} $ to  $ t $ and  applying Theorem \ref{t: Pym} with the $ st $-paths  
obtained 
from $ \mathcal{W} $ and from the initial segments of
 the paths in $ \mathcal{P} $ we conclude:

\begin{cor}\label{cor: Pym cor}
Let $ \mathcal{P} $ be a system of edge-disjoint $ st $-paths. Then there is a large $ st $-wave $ 
\mathcal{W} $ with $ \delta_{\mathcal{W}}(s)\supseteq \delta_{\mathcal{P}}(s)$.
\end{cor}

Finally, we will make use of the following classical lemma (see Lemma 3.3.2 and 3.3.3 in \cite{diestel2017graph}):
\begin{lem}[Augmenting path lemma]\label{l: aug path}
Assume that $ G $ is a (possibly infinite)  graph, $ s,t\in V(G) $ are distinct and $ \mathcal{P} $ is a system of edge-disjoint $ 
st $-paths in $ G $. Then either there exists an $ st $-cut $ C $ orthogonal to $ 
\mathcal{P} $ (i.e. $ \mathcal{P} $ is as in Theorem \ref{t: Inf Menger}) or there is another system $ \mathcal{Q} $ of 
edge-disjoint $ st $-paths for which 
$ \delta_\mathcal{Q}(s) \supset   \delta_\mathcal{P}(s) $ with $ \left|\delta_\mathcal{Q}(s) \setminus   \delta_\mathcal{P}(s) 
\right|=1$   and $ \delta_\mathcal{Q}(t) 
\supset  \delta_\mathcal{P}(t) $ with $ \left|\delta_\mathcal{Q}(t) \setminus   \delta_\mathcal{P}(t) 
\right|=1$.
\end{lem}

All the definitions and results in the section remain valid (but might sound less natural)  if $ s $ and $ t $ are not 
vertices but disjoint vertex sets.

\section{The proof of the main result}\label{sec: proofLovCHer}
We start by giving a short outline of the proof. In the first two subsections we apply relatively simple techniques in order to 
reduce Theorem \ref{t:LCh inf} to Theorem \ref{t: LCh1 infIntro}  and  cut the latter problem into countable sub-problems. The 
third subsection is devoted to 
the proof of the reduced problem, namely the countable case of Theorem \ref{t: LCh1 infIntro}.  The core of that proof is 
to show that for every given  $ e\in  \bigcup_{t\in T}\delta(t)$  there is a path $ P $ through $ e $ such 
that $ G-E(P) $ maintains the premise of Theorem \ref{t: LCh1 infIntro}. If $ G $ is countable, then one can simply use this 
recursively  to build the desired path-system. 
\begin{proof}[Proof of Theorem \ref{t:LCh inf}]
We will  use only that $ \{ t\in T: d(t)>1 \} $ is countable instead of the countability of the whole $ T $. As a first step we reduce 
Theorem \ref{t:LCh inf} to the following theorem.

\begin{thm}\label{t:LCh1 inf}
Let $G$ be a  graph and let $T\subseteq V(G)$ be  such that $ d(t)\leq 1 $ for all but countably many $ t\in T $ and 
there is no $ X\subseteq V(G)\setminus T $ 
where $ d(X) $ is an odd natural number. Assume that for each $t\in T$  
there is a system $\mathcal{P}_t$ of edge-disjoint $T$-paths covering $\delta(t)$. Then there exists a system $\mathcal{P}$ of 
edge-disjoint $T$-paths covering $\bigcup_{t\in T}\delta(t)$. 
\end{thm}

For $ s\in T $ we will write shortly $ s $-wave instead of $ s(T-s) $-wave. Recall, it is a  system $\mathcal{W}$ of  pairwise 
edge-disjoint paths starting at $ s $ such that the set $ C_{\mathcal{W}} $ of the last edges of the paths is a minimal cut 
between $ s $ 
and $ T-s $.
\subsection{Elimination of waves}\label{subs: waveelim}

We will call shortly the condition about the existence of the path-system $\mathcal{P}_t$ in Theorem \ref{t:LCh1 inf} the 
\emph{linkability condition for} $t$ (w.r.t. $ G $ and $ T $)
and we refer to the conjunction of these for $t\in T$ as the
\emph{linkability condition}. First we define a process that we call \emph{wave elimination}. We may assume that $ G $ is 
connected otherwise we define the elimination process component-wise. Let $ T'\subseteq T $ be given. We pick an arbitrary 
enumeration $ T'=\{ t_\xi: \xi<\kappa \} $ 
and  define by transfinite 
recursion $ G_\xi $ for $ \xi\leq \kappa $. Let $ G_0:=G $. If $ 
G_\xi $ is already defined then let $ \mathcal{W}_{\xi} $ be a large $ t_\xi $-wave with respect to  $ G_\xi $ and $ T $ 
(exists by Lemma \ref{l: large wave}). We obtain $ G_{\xi+1} $ by contracting  the $ t_\xi $-side of the cut $ 
C_{\mathcal{W}_{\xi}} $ in $ G_\xi $ to $ t_\xi $ (see Figure \ref{fig welim}). If $ \xi $ is a limit ordinal then 
we obtain $ G_{\xi} $ by doing all the previous contractions simultaneously. The recursion is done. 

The cardinal $ 
d_{G_{\kappa}}(X) $ for 
$X\subseteq 
V(G_{\kappa})\setminus T $ cannot be an odd natural number  because $ d_{G_{\kappa}}(X)=d_G(X) $ and $ G $ was inner 
Eulerian w.r.t. $ T $. Furthermore, Corollary \ref{cor: delta(s) coverable} ensures that  for $ \xi<\kappa $
there is no non-trivial $ t_\xi $-wave 
in $ G_{\xi+1} $.  Since any $ t_\xi $-wave in $ 
G_\kappa$ is corresponding to a $ t_\xi $-wave in $  G_{\xi+1} $, it follows that for each  $ t\in T' $ there is only the 
trivial 
$ t $-wave in $ G_\kappa $. By taking $ T':=T $, this is (more than) enough to guarantee the 
 linkability condition at Theorem \ref{t:LCh1 inf} (see Corollary \ref{cor: delta(s) coverable}).
 Therefore $ G_\kappa $ satisfies the premise of Theorem \ref{t:LCh1 inf} and hence assuming Theorem \ref{t:LCh1 inf} we 
 may conclude that there is a 
system $ \mathcal{P} $ of $ T $-paths in $ G_\kappa $ covering $ 
\bigcup_{t\in T}\delta_{G_\kappa}(t) $. By using the waves $ \mathcal{W}_{\xi} $, the system $ \mathcal{P} $ can be 
extended 
to a system $ \mathcal{Q} 
$ of $ T $-paths in $ G $ where the $ t_\xi(T-t_\xi) $-cut $ C_{\mathcal{W}_{\xi}} $ is orthogonal to $ \mathcal{Q}_{t_\xi}:=\{ 
Q\in 
\mathcal{Q}: t_\xi\in V(Q) \} $. Therefore $ \mathcal{Q} $ satisfies the requirements of Theorem \ref{t:LCh inf}.

\begin{figure}[H]
\centering
\begin{tikzpicture}[scale=0.6]

\node (v1) at (-2,-0.5) {};
\node  at (-2.2,-0.7) {$t_0$};

\draw  (v1) ellipse (2 and 2);
\node (v3) at (4.5,-0.5) {};
\node  at (4.3,-0.7) {$t_1$};
\node (v2) at (0.5,5.5) {};
\node  at (0.5,5.25) {$t_2$};
\draw  (v2) ellipse (2 and 2);
\draw  (v3) ellipse (2 and 2);

\draw  plot[smooth, tension=.7] coordinates {(v1) (-2,0) (-2,0.5) (-2.5,1) (-2.5,2)};
\draw  plot[smooth, tension=.7] coordinates {(-2,-0.5) (-1.5,0) (-1.5,1) (-1.5,2)};
\draw  plot[smooth, tension=.7] coordinates {(-2,-0.5) (-1,-0.5) (-1,0.5) (-0.5,1.5)};
\draw  plot[smooth, tension=.7] coordinates {(-2,-0.5) (-1.5,-1) (-0.5,-1) (0,-0.5) (0.5,0)};
\draw  plot[smooth, tension=.7] coordinates {(v2) (-1,5) (-0.5,4.5) (-1,4) (-1,4)};
\draw  plot[smooth, tension=.7] coordinates {(0.5,5.5) (0,5) (0,4.5) (0.5,4.5) (0,4) (-0.5,3.5)};
\draw  plot[smooth, tension=.7] coordinates {(0.5,5.5) (1,5) (1.5,5.5) (1.5,4.5) (1.5,3.5)};
\draw  plot[smooth, tension=.7] coordinates {(v3) (4.5,0.5) (4,0.5) (3,0) (2,0)};
\draw  plot[smooth, tension=.7] coordinates {(4.5,-0.5) (5,0) (5,0.5) (5.5,1.5)};
\draw  plot[smooth, tension=.7] coordinates {(4.5,-0.5) (3.5,-0.5) (3,-1) (2,-1)};
\draw  plot[smooth, tension=.7] coordinates {(4.5,-0.5) (4.5,-1) (4,-1.5) (3,-2) (2.5,-2)};
\node at (-0.5,0) {$\mathcal{W}_{0}$};
\node at (0.9,4.1) {$\mathcal{W}_{2}$};
\node at (3.6,-1.1) {$\mathcal{W}_{1}$};
\draw  plot[smooth, tension=.7] coordinates {(-2,-0.5) (-2.5,-0.5)};
\draw  plot[smooth, tension=.7] coordinates {(-2,-0.5) (-2,-1)};
\draw  plot[smooth, tension=.7] coordinates {(0.5,5.5) (0,6)};

\draw  plot[smooth, tension=.7] coordinates {(0.5,5.5) (1,6)};
\draw  plot[smooth, tension=.7] coordinates {(0.5,5.5) (0.5,6)};
\draw  plot[smooth, tension=.7] coordinates {(4.5,-0.5) (5,-1)};
\draw  plot[smooth, tension=.7] coordinates {(4.5,-0.5) (5,-0.5)};
\end{tikzpicture}
\caption{The contracted vertex sets during the wave elimination}\label{fig welim}
\end{figure}
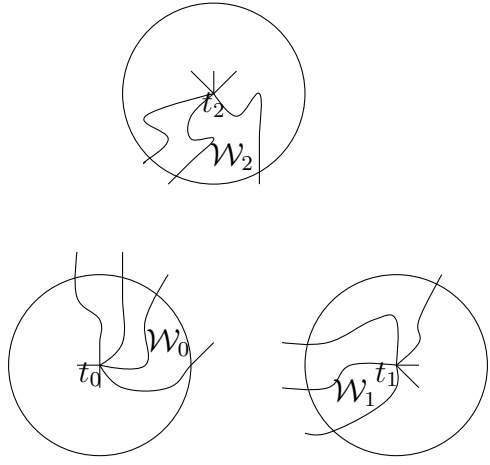 

\subsection{Reduction to countable graphs}\label{subs: countable}
In the next reduction we show that it is enough to restrict our attention  to countable graphs in the proof of Theorem 
\ref{t:LCh1 inf}.  First of all, we may assume without loss of generality that $ 
T $ does not span any edges. Indeed,  otherwise we consider the graph $ G' $ obtained from $ G $ via the deletion of  
 those edges. Then we pick a path-system $ \mathcal{P}' $ by applying this special case of Theorem \ref{t:LCh1 inf} with $ G'  $ 
 and $ T $. Finally,  we obtain $ 
 \mathcal{P} $ by extending  $ \mathcal{P}' $  with the 
deleted edges as $ T $-paths of length one.  

By applying some  basic elementary submodel-type  arguments we cut $ E $ into countable pieces each of them satisfying both 
the inner Eulerian and the linkability condition w.r.t. $ T $.  The contraction of $ T $ to an arbitrary $ t $ results in an Eulerian 
graph $ 
G/T $ by Theorem \ref{t: NW}  thus we can take
a partition $\mathcal{O}$  of $E(G/T)=E$ into (edge sets of)  $ G/T $-cycles. These are cycles and $ T $-paths in $ G $ .
Let $ T':=\{ t\in T: d(t)>1 \} $. For $t\in T'$ let $ \mathcal{P}_t$ be a system 
of $T$-paths witnessing the linkability condition for $ t $ and let 
$ \mathcal{\mathcal{E}}:=\{ E(P): (\exists t\in T') P\in \mathcal{P}_t \}$. We 
define a 
closure operation $c$ on $2^E$ in the following way. Intuitively we want to close a set $ F_0\subseteq E $ under the property 
that if it shares 
an edge with some $ O \in \mathcal{O}$ or $ E(P)\in \mathcal{E} $, then it contains it completely.
Formally let $c(F_0):=\bigcup_{n\in \mathbb{N}}F_n$ where 

\[ F_{n+1}:=F_n\cup \bigcup\{O\in \mathcal{O}: F_n\cap O\neq \emptyset\}\cup \bigcup \{ E(P)\in \mathcal{E}:  F_n\cap 
E(P)\neq 
\emptyset\}.  \]

We call an $F$ $c$-closed if $c(F)=F$.  We claim that $c$ satisfies the following properties:

\begin{enumerate}
    \item\label{i:c1} The family of $c$-closed sets forms a complete Boolean algebra with respect to  the usual $\cup$ and $\cap$;
    \item\label{i:c2} If $F$ is countable then so is $c(F)$;
    \item\label{i:c3} If $F$ is $c$-closed, then  $(V,F,I \upharpoonright F)$ and $ T $ satisfy the  premise of Theorem 
    \ref{t:LCh1 inf}.
    
\end{enumerate}
Indeed,  property (\ref{i:c1}) follows directly from the construction and  (\ref{i:c2}) holds because of the assumption $ 
\left|T'\right|\leq \aleph_0 $ and the fact that each edge $ e $ is  used by at most one path in $ \mathcal{P}_t $ for every fixed $ t 
$ and $ e $ is contained in a unique element of $ \mathcal{O} $. The `inner Eulerian' and linkability 
 for $ t\in T' $ in condition (\ref{i:c3}) are ensured by $ F $ not subdividing any $ O $ and $ E(P) $ respectively.  Recall that $ 
 d(t)\leq 1 $ for  $ t\in T\setminus T' $ by 
definition. Preservation of the linkability for these $ t $ is ``automatic'':

\begin{lem}\label{l: t-leaf}
If $ H $ is an inner Eulerian graph w.r.t. $ T\subseteq V(H) $, then the linkability condition holds for all $ t\in T $ with $ d(t)\leq 1 
$.
\end{lem}
\begin{proof}
$ E(H) $ can be partitioned into the edge sets of cycles and $ T $-paths. If $ d(t)=1 $, then the unique edge incident with $ t $ 
cannot be in a cycle so must be in a $ T $-path.
\end{proof}

In order to reduce Theorem \ref{t:LCh1 inf} to countable graphs, it is enough to partition $E$ into countable $c$-closed sets 
$F_\xi$. Indeed,  then $G_\xi:=(V,F_\xi, I\upharpoonright F_\xi)$ is countable (apart from 
isolated vertices) and satisfies the  premise of Theorem \ref{t:LCh1 inf} with $ T $ by property (\ref{i:c3}). Hence by applying 
the countable case 
of Theorem \ref{t:LCh1 inf}, we can take a system $\mathcal{P}_\xi$ of 
$T$-paths in $ G_\xi $ covering  the edges $ \bigcup_{t\in T}\delta_{G_\xi}(t) $. Finally, $\bigcup_\xi \mathcal{P}_{\xi}$ is as 
desired. 

Suppose that the 
pairwise disjoint countable $c$-closed sets 
$\{F_\xi: \xi< \alpha\}$ are already defined for some ordinal $\alpha$. Then $E\setminus \bigcup_{\xi<\alpha}F_\xi $ is 
$c$-closed by property (\ref{i:c1}). If it is empty then we are done. Otherwise let $F_\alpha:= c(\{e\})$ for an arbitrary $e\in 
E\setminus \bigcup_{\xi<\alpha}F_{\xi} $, which is countable by property (\ref{i:c2}). The recursion is done.

\subsection{The proof of Theorem \ref{t:LCh1 inf}}
We will make use of the following simple observation. 
\begin{obs}\label{o: 1path delete}
 The deletion of the edges of a $ T $-path preserves the condition that there is no $ X\subseteq V\setminus T $ with $ d(X) $ odd.
\end{obs}
The core of our proof is the repeated application of the following claim:

\begin{claim}\label{cl: 1path make}
Let $G$ be a  graph and let $T\subseteq V(G)$ be such that $ G $ is inner Eulerian w.r.t. $ T $ (i.e. there is no $ X\subseteq 
V(G)\setminus T $ 
where $ d(X) $ is an odd natural number). Assume that for each $t\in T$  
there is a system $\mathcal{P}_t$ of edge-disjoint $T$-paths covering $\delta(t)$. Then for every $t\in T$ and $e\in \delta(t)$ 
there exists a $T$-path $P$ through $e$ such that $G-E(P)$ satisfies the linkability 
condition (and remains inner Eulerian w.r.t. $ T $).
\end{claim}

Indeed, we only need to prove Theorem \ref{t:LCh1 inf} for countable $G$ as discussed in the previous subsection. Assuming 
Claim \ref{cl: 1path make},  a system of $T$-paths covering $ \bigcup_{t\in T}\delta(t) $ can be 
constructed by a straightforward recursion. 

\begin{nbiz}[Proof of Claim \ref{cl: 1path make}]
First we give a proof in the special case where there is some $s\in T$ such that $d(t)\leq 1$ for all $t\in T-s$. Let us fix a system 
$\mathcal{P}_s$ of edge-disjoint paths between $s$ and $T-s$ covering $\delta(s)$. 

For $e\in \delta(s)$, we simply take the unique $P\in \mathcal{P}_s$  through $e$. By Observation 
\ref{o: 1path delete}, the graph $ G-E(P) $ is still inner Eulerian w.r.t. $ T $. By Lemma \ref{l: t-leaf} it is enough to check that 
the 
linkability condition is preserved for $ s $ but it is 
obviously true  witnessed by $ 
\mathcal{P}_s\setminus \{ P \} $. 

Suppose now that $e\in \delta(t)$ for a $t\in T-s$. If $t$ is an end-vertex of some $P\in \mathcal{P}_s$, then we take $P$ and 
argue as in the previous paragraph. If it is not the case, then either we  replace $ \mathcal{P}_s $ by another $ \mathcal{P}_s' $ 
where $t$ is an end-vertex of some $P\in \mathcal{P}_s'$ or choose $ P $ to be edge-disjoint from $ \mathcal{P}_s $. To do so,  
let $Q$ be an 
arbitrary path between $t$ and $T-t$. If $E(Q)\cap E(\mathcal{P}_s)= \emptyset$, then we take 
$P:=Q$ and the linkability condition holds for $s$ since $\mathcal{P}$ lives in $G-E(P)$.
If $E(Q)\cap E(\mathcal{P})\neq \emptyset$, then let $v\in V(Q)\cap V(\mathcal{P})$ be the first common vertex while going 
along $ Q $ from $t$. 
Let $P'\in 
\mathcal{P}_s$ be such that $v\in V(P')$. We get $\mathcal{P}_s'$ by replacing $P'$ in $\mathcal{P}_s$ with the path $ P $ we 
obtain by uniting the initial segment of $ P' $ from $ s $ to $ v $ with the initial segment of $ Q $ from $ t $ to $ v $. 

Applying this iteratively we conclude:
\begin{cor}\label{cor: SpecialC}
Theorem \ref{t:LCh1 inf} holds whenever there is an  $s\in T$ such that  $d(t)\leq 1$ for every $ t\in T-s $.
\end{cor}

\begin{prop}\label{p: key}
Assume that $ G=(V,E,I) $ is an inner Eulerian graph w.r.t. $ T\subseteq V $ and there is an $ s\in T $ such that there is no 
non-trivial $s$-wave. Then for every $f, h\in E $,  the linkability condition holds for $ s $ in $G-f-h$.
\end{prop}
\begin{proof}
We may assume without loss of generality that $G$ is connected, since only the component containing $ s $ is relevant. Since 
deletion of edges in $ \delta(s) $ makes the linkability for $ s $ a weaker requirement, we can also assume that $f, h\in E \setminus 
\delta(s)$. If $G$ is 
finite and $X\subseteq V$ with $X\cap T=\{s\}$, then 
$d(s)$ and $d(X)$ must have the same parity because $d(v)$ is even for $v\in X-s$. This observation of Lovász led immediately 
to the justification of Proposition
\ref{p: key} for finite graphs. Indeed, on the one hand, $ d(s)<d(X) $ if $ \{ s \}\subsetneq X \subseteq V\setminus (T-s) $, since 
$ \delta(s) $ is the only Erdős-Menger $ s(T-s) $-cut by assumption. On the other hand, the same 
parity of $ d(s) $ and $ d(X) $   ensures
$d(s)+2 \leq d(X)$.  The proof of Proposition \ref{p: key} for infinite graphs is more involved and we need 
some preparation. 

For a graph $H$ and distinct $s,t\in V(H)$, we call an 
 Erdős-Menger $st$-cut $C$ $s$\emph{-tight} if there is system $\mathcal{P}$ of edge-disjoint paths in $H$ between $s$ and 
 $t$ covering $\delta_H(s)$ and  every such a path-system is orthogonal to $ C $.

\begin{lem}\label{l:tight cut}
Assume that $H$ is a graph, $s,t\in V(H)$ are distinct and there is a system $\mathcal{P}$   of edge-disjoint paths in $H$ 
between $s$ and $t$ covering $\delta_H(s)$ and  there is an $e\in E(H)\setminus 
\delta_H(s)$  such that $ e\in E(\mathcal{P}) $ for every such  path-system. 
Then there exists an $s$-tight Erdős-Menger $ st $-cut $C$  containing $e$.
\end{lem}
\begin{proof}
We may assume that $ H $ is connected, since otherwise we consider the component containing $ s $ and $ t $. Let 
$\mathcal{P}$  and $ e $ be as in the lemma. Then there is a unique $ P_e\in \mathcal{P} $ through $ e $. If  $H-e  $ is 
disconnected, 
then  we must have $  \mathcal{P}=\{ P_e \} $ and the cut  $ C:=\{ e \} $ is as desired. Suppose that $H-e $ is connected. Let $ 
D $ be 
the $ \preceq 
$-smallest 
Erdős-Menger $ st $-cut in $ H-e $ (see Lemma \ref{l: EMlattice}). We are going to prove that $ C:=D+e $ is as desired. To 
do so, it is enough to show that 
 $ \mathcal{Q}:= \mathcal{P}\setminus \{ P_e \} $  is orthogonal to $ D $. Indeed, if this holds, then $ e $ 
must connect the two parts of  the cut $ D $ in $ H-e $  and therefore $ D+e $ is an $ st $-cut in $ H $ and  $ \mathcal{P} $ is  
orthogonal to it. 

Suppose for a contradiction that $ \mathcal{Q} $  is not orthogonal to $ D $. 

\begin{figure}[H]
\centering
\begin{tikzpicture}
\node (v1) at (-2,0) {$s$};
\draw  (0,1.4) rectangle (2.8,-1.6);
\node (v2) at (2,0) {$t$};
\node at (-0.5,1) {$D$};
\draw  plot[smooth, tension=.7] coordinates {(v1) (-1.6,-0.8) (-0.6,-1.2) (0.4,-1.2) (0.4,-0.6) (-0.4,-0.4) (0.4,0) (1,-0.4) (v2)};
\draw  plot[smooth, tension=.7] coordinates {(-2,0) (-1,0.2) (-0.4,0) (0.4,0.6) (1.2,0.4) (2,0)};
\draw  plot[smooth, tension=.7] coordinates {(-2,0) (-1.8,0.6)};
\node (v4) at (0.5,-0.8) {};
\node (v3) at (-0.4,-0.8) {};
\draw  (v3) edge (v4);
\node (v5) at (-0.4,0.6) {};
\node (v6) at (0.5,1.2) {};
\draw  (v5) edge (v6);
\node at (0.8,0.2) {$\mathcal{Q}$};
\node at (-1.6,0.4) {$e'$};
\end{tikzpicture}
\caption{Graph $ H-e $ where $ D $ is not orthogonal to $ \mathcal{Q} $. The first edge of $ P_e $ is $ e' $.}\label{fig tight}
\end{figure}
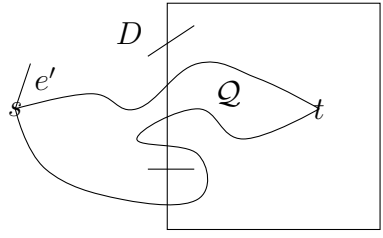 
Let $ H' $ be the graph we get by contracting the $ t $-side 
of $ D $ to $ t $ in $ H $. Then $ 
D=\delta_{H'-e}(t) $ and it is the  only element of $ \mathfrak{C}_{H'-e}(s,t) $ since it is the smallest but also the largest 
one. 
We apply the  Augmenting path lemma \ref{l: aug path} in $ H'-e $ with $ s,t $ and the set $ \mathcal{Q}' $ of $ st $-paths in 
$ 
H'-e $ given by the initial segments of the paths in $ \mathcal{Q} $. The augmentation must be 
successful, since otherwise it would give a $ D'\in \mathfrak{C}_{H'-e}(s,t) $ with $ D'\neq D  
$. Indeed, 
$ D\setminus E(\mathcal{Q}') \neq \varnothing$ by the indirect assumption but $ D'\subseteq E(\mathcal{Q}')$ according to 
Lemma \ref{l: aug path}.  The successful augmentation provides a system 
$ \mathcal{Q}'' $ of edge-disjoint $ st $-paths in $ H'-e $ covering $ \delta_{H'-e}(s) $. Indeed,  there is a unique $ e'\in 
\delta_{H'-e}(s) $ which is uncovered by $ \mathcal{Q}' $, namely the first edge of $ P_e $, but the  Augmenting path lemma 
\ref{l: aug path} ensures $\delta_{\mathcal{Q}'}(s)  \subset\delta_{\mathcal{Q}''}(s)  $.  Since $ D\in 
\mathfrak{C}_{H-e}(s,t) $, the 
paths in $ \mathcal{Q}'' $ can be forward extended in $ H $ to obtain a system 
 of edge-disjoint $ st $-paths in $ H-e $ covering $ \delta_H(s) $  contradicting the obligatory usage of $ e $ in the assumption of 
 the lemma. 
\end{proof}

Since the only Erdős-Menger $s(T-s)$-cut is $ \delta(s) $ (see Corollary \ref{cor: delta(s) coverable}) and $f\notin \delta(s)$, 
Lemma 
\ref{l:tight cut} (applied in $ G/(T-s) $) ensures that there is a system 
$\boldsymbol{\mathcal{P}_s}$ of edge-disjoint paths in $G-f$ between $s$ and $T-s$ covering $\delta(s)$. Suppose for a 
contradiction that 
such a path-system cannot be found in $G-f-h$. By applying Lemma \ref{l:tight cut} again this time with $G-f$ and $ h $, we 
obtain 
an 
$s$-tight Erdős-Menger
$s(T-s)$-cut $C$ in $ G-f $ containing $h$. Let $S$ be the $s$-side of the cut $ C $. Then  $\delta_{G-f}(S)=C$ and we must 
have 
$f\in 
\delta_{G}(S) $ since otherwise the initial segments of the paths in $\mathcal{P}_s$ up to their unique edge in $C$ would form a 
non-trivial $s$-wave with respect to $G$ and $T$.  Thus $ f, h\in \delta(S)\setminus \delta(s) $.   We  
define $ G' $ by
extending $G[S]$ with new vertices $\{t_e: e\in \delta_G(S)\}$ and with the edges $ \delta(S) $ where an $ e\in \delta(S) $ keeps 
its original end-vertex in $ S $ and gets $ t_e $ as the other end-vertex. Let $T':=\{s\}\cup \{t_e: e\in \delta_G(S)\}$.

\begin{figure}[H]
\centering
\begin{tikzpicture}

\draw  (0.8,0) node (v10) {} ellipse (2.2 and 1.8);
\node (v9) at (0.8,-1) {$s$};
\node (v1) at (-1.6,1.8) {$t_f$};
\node (v2) at (-0.6,0.6) {};

\node (v3) at (-0.2,2.4) {$t_{e_0}$};
\node (v4) at (0.2,1) {};
\node (v5) at (0.6,2.6) {$t_{e_1}$};
\node (v6) at (1,1.2) {};
\node (v7) at (1.8,2.6) {$t_{e_2}$};
\node (v8) at (1.8,1) {};

\node at (0.2,-0.2) {$Q_f$};
\node at (0.6,1) {$ Q_{\neg s} $};

\draw  plot[smooth, tension=.7] coordinates {(v9) (v10) (v2) (v1)};
\draw  plot[smooth, tension=.7] coordinates {(v3) (0,1.4) (0,1) (0,0.6) (0.4,0.6) (0.8,0.6) (v6) (0.8,1.6) (v5)};
\draw  plot[smooth, tension=.7] coordinates {(v9) (1.4,-0.8) (1.8,-0.2) (1.6,0.4) (1.6,1.2) (v7)};

\node at (-1.8,0) {$S$};
\node at (-0.8,1.8) {$C+f$};
\end{tikzpicture}
\caption{Graph $ G' $ and path-system $ \mathcal{Q} $.}\label{fig inc}
\end{figure}
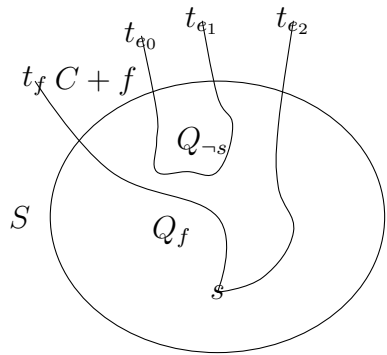 

 For $ 
X\subseteq V(G')\setminus T' $, the cardinal  $ 
d_{G'}(X) $ 
cannot be an odd number  because $ 
d_{G'}(X)=d_G(X) $ by construction. 
Moreover, the linkability condition with respect to $ G' $ and $ T' $ holds, since for $ s $ it is witnessed by the initial segments of 
the paths in   $\mathcal{P}_s$ while the connectivity of  $ G' $ guarantees it for the vertices in $  T'-s $. Thus the premise of 
Theorem 
\ref{t:LCh1 inf} are satisfied, furthermore, every vertex in $T'$ except possibly $s$ has degree $ 1 $. By applying Corollary 
\ref{cor: SpecialC} to $ G' $ and $ 
T' $, we can take a system $\boldsymbol{\mathcal{Q}}$ of $T'$-paths 
in $G'$ covering all the edges $ \bigcup_{t\in T'}\delta_{G'}(t) $. 

It cannot happen that all $ Q\in \mathcal{Q} $ have $ s $ as an 
end-vertex because then $\mathcal{Q}$ would provide a non-trivial $ s $-wave with respect to $ G $ and $ T $ (where $ f\notin 
\delta(s) $ is used to ensure `non-trivial').    Let $ Q_{\neg 
s}\in 
\mathcal{Q} $ be a path with $ s\notin V(Q_{\neg s}) $  and let us denote the end-vertices of $ Q_{\neg s} $ by $\boldsymbol{ 
t_{e_0}} $ and 
$  \boldsymbol{t_{e_1}} $. 

\begin{lem}\label{l: icomp}
Every 
system $ \mathcal{R}' $ of edge-disjoint $ T' $-paths in $ G' $ covering $ \delta(s) $ and avoiding $ t_f $ must use all the 
vertices $ \{t_e: e\in C\}  $. 
\end{lem}
\begin{proof}
Suppose for a contradiction that $ \mathcal{R}' $ is a counterexample. We may assume that each path in $ \mathcal{R}' $ starts 
at $ s $ since otherwise we remove the rest. Let $ \mathcal{R} $ be the path-system in $ G-f $ 
corresponding to $ 
\mathcal{R}' $. Then $ \mathcal{R} $  is a system of edge-disjoint paths starting at $ s $ and having exactly their last edges in $ 
C $ such that $ C\setminus E(\mathcal{R})\neq \varnothing$. Since $ C $ is an Erdős-Menger $ s(T-s) $-cut in $ G-f $, the 
paths in 
$ \mathcal{R} $ can be forward extended to obtain a system $ \mathcal{R}^{+} $ of edge-disjoint $ s(T-s) $-paths with 
$ C\cap E( \mathcal{R}^{+})=C\cap E( \mathcal{R}) $. Then $ \mathcal{R}^{+} $ also covers $ \delta(s) $  and $ 
C\setminus E( \mathcal{R}^{+})\neq \varnothing $ contradicting the $ s $-tightness of $ C $ in $ G-f $.
\end{proof}

   There must be a $ Q_f\in \mathcal{Q} $  with $ t_f\in V(Q_f) $ since otherwise $\mathcal{R}'= \mathcal{Q}-Q_{\neg s} $ 
   contradicts Lemma \ref{l: icomp}. We claim that the other end-vertex of $ Q_f$  
   must be $ s 
   $,  thus in particular $ Q_f\neq Q_{\neg s} $ and hence $ f\notin \{ e_0, e_1 \} $. Indeed, since otherwise the system  
   $\mathcal{Q}_s:=\{Q\in \mathcal{Q}: s\in 
   V(Q)\}$ of 
edge-disjoint $ T' $-paths in $ G' $ covers $ \delta(s) $ using neither $ t_f $ nor the other end-vertex of $ Q_f $ which contradicts 
Lemma \ref{l: icomp}. Now we consider the path-system $ \mathcal{Q}_s \setminus \{ Q_f \} $. It covers all but one edges in $ 
\delta(s) $ and 
avoids $ e_0 $ and $ e_1 $.  We apply the Augmenting path lemma \ref{l: aug path} in $ G' $ with $ \mathcal{Q}_s \setminus \{ 
Q_f \}, s 
$ and  $ \{ 
t_e: e\in C \} $. If the augmentation is successful, then the resulting path-system covers $ \delta(s) $ and at least one of $ e_0 $ 
and $ e_1 $ is 
still unused contradicting Lemma \ref{l: icomp}. Thus the Augmenting path lemma ensures that 
we can pick a single edge from each path in $ \mathcal{Q}_s \setminus \{ Q_f \} $ such that the resulting edge set $ C' $ 
separates $ s $ and $ \{ 
t_e: e\in C \} $ in $ G' $. We take the initial segments of the paths in $ \mathcal{P}_s $ until the first meeting with $ C' $ and 
continue them 
forward using the terminal segments of the corresponding paths from $ \mathcal{Q}_s \setminus \{ Q_f \} $ to obtain a set of $ T' 
$-paths in $ G' $ covering $ \delta(s) $ without using $ t_{e_0}, t_{e_1} $ and $ t_f $, which contradicts Lemma \ref{l: icomp}.
\end{proof}

Now we can finish the proof of Claim \ref{cl: 1path make}. Suppose for a contradiction that $G, T, s\in T$ and $e_0\in \delta(s)$ 
form a 
counterexample and $\mathcal{P}_s=\{P_e: e\in \delta(s)\}$ is a system of edge-disjoint $T$-paths with $e\in E(P_e) $. We 
assume that $G, T, s, e_0$ and  $\mathcal{P}_s$ have been chosen to minimize $\left|E(P_{e_{0}})\right|$ among the possible 
options. We apply wave 
elimination with $ 
T-s $ choosing  each wave according to Corollary \ref{cor: Pym cor}. Consider the resulting $ G' 
$ and observe that $ G', T, s, e_0 $ must be also a counterexample (see Figure \ref{fig welim}). For $ e\in \delta(s)$, let  $ 
P'_e $ be the longest initial 
segment of $ P_e $ from $ s $ that lives in $ G' $. Then we must have $ P_{e_0}=P'_{e_0} $  since otherwise    
$\mathcal{P}'_s=\{P'_e: 
e\in \delta(s)\}$ and $ \left|E(P'_{e_0})\right|<\left|E(P_{e_0})\right| $ contradicts the choice  of $ G, T, s $ 
and $ e_0 $. 
Note that $\left|E(P_{e_{0}})\right| \geq 2$ 
because if $P_{e_{0}}$  consisted of  the 
single edge 
$e_0$, then $P:=P_{e_{0}}$ would satisfy Claim \ref{cl: 1path make} for $ G', s$ and $ e_0 $. 

 Let $f_0\in 
E(P_{e_0})$ be the edge right after 
$e_0$ in $ P_{e_0} $. We replace in $ G' $ the edges $e_0$ and $f_0$ by a single new edge $h_0$ connecting $s$ 
and the end-vertex of $f_0$ that is not shared with $e_0$ (splitting technique by Lovász from \cite{lovasz1976some}). Let 
$P_{h_0}$ be the path in the 
resulting graph $ G'' $ with $E(P_{h_0})=E(P_{e_0})-e_0-f_0+h_0$ and let us define 
$\mathcal{P}_s'':=\mathcal{P}_s'-P_{e_0}+P_{h_0}$.
 For $ X\subseteq V\setminus T $ the quantities $ d_{G'}(X) $ and $ d_{G''}(X) $ are either both infinite or they have the 
 same 
 parity, 
 thus $ G'' $ is also inner Eulerian w.r.t. $ T $. The linkability condition for $ s $  in $ G'' $ is witnessed by $ \mathcal{P}_s'' 
 $. 
 Let $ 
 t\in T-s $ be arbitrary. The linkability condition for $ t $ holds  in $ G'-e_0-f_0 $ by 
 Proposition \ref{p: key}, moreover, if $ h_0\in \delta_{G''}(t)  $, then $ h_0 $ is an edge between $ s $ and $ t $  and hence 
 a $ T 
 $-path  itself. Thus the linkability condition holds in $ G'' $.
Note that $G'', T, s$ and $h_0$ cannot be a counterexample for Claim \ref{cl: 1path make} 
because $\left|E(P_{h_{0}})\right|=\left|E(P_{e_{0}})\right|-1 $. Therefore we can pick some $ T $-path $ P'' $ in $ G'' $ 
through 
$ h_0 $ such that the linkability condition holds in $ G''-E(P'') $. Let us take then a $ T $-path $ P' $ in $ G' $ through 
$ e_0 $ with $ E(P)\subseteq E(P'')-h_0+e_0+f_0 $. Since  $ G''-E(P'') $ is a subgraph of $ G'-E(P') $ and $ 
\delta_{G'-E(P')}(t)=\delta_{G''-E(P'')}(t) $ holds for $ t\in T $, the linkability condition in $ G''-E(P'') $ implies the linkability 
in $ G'-E(P') $. This contradicts the fact that $ G', T, s $ and $ e_0 $ form a counterexample for Claim \ref{cl: 1path make}.
\end{nbiz}
\end{proof}

\section{Open questions}\label{sec: oultook}
First of all, we expect that in Theorem \ref{t:LCh inf} the restriction about the size of $ T $ can be completely omitted:  
\begin{conj}\label{conj: LCh}
Let $G$ be a  graph and let $T\subseteq V(G)$   such that  there is no $ X\subseteq V(G)\setminus 
T $ 
where $ d_G(X) $ is an odd natural number. Then there exists a system 
$\mathcal{P}$ of edge-disjoint $T$-paths such that for every $t\in T$: one can choose exactly one edge from each $P\in 
\mathcal{P}$ having $t$ as an end-vertex in such a way that the resulting edge set $C$ is a cut separating $t$ and $T-t$. 
\end{conj}

We conjectured already in the Introduction (Conjecture \ref{conj: strongly maxim}) the existence of  strongly maximal systems 
 of $ T $-paths with different concepts of disjointness. We believe that strong maximality can be characterized by the existence of 
 a certain dual object reflecting the corresponding classical theorems of Gallai \cite{gallai1964maximum} and Mader 
 \cite{mader1978maximalzahl, mader1978maximalzahlH}.

\subsection{Edge-disjoint T-paths in not necessarily inner Eulerian graphs}
Let $ G $ be a  graph and let $ T\subseteq V(G) $. A $ T $\emph{-partition} is a family $ \mathcal{A}=\{ X_t: t\in T \} $ of 
pairwise 
disjoint subsets of $ V(G) $ such that $ X_t\cap T=\{ t \} $.  If $ G $ is finite, then we call a component $ Y $ of $ G-\bigcup 
\mathcal{A} $ \emph{obstructive} if $ d(Y) $ is odd.  
Let $ \mathsf{o}(G, \mathcal{A}) $ 
be the number of the obstructive components.
\begin{thm}[Mader, \cite{mader1978maximalzahl}]\label{t: Mader T-path}
Let $ G $ be a finite graph and let $ T\subseteq V(G) $.   Then the 
maximal number of pairwise edge-disjoint $ T $-paths is

\[\min \left\lbrace \frac{1}{2}\left( \sum_{t\in T}d(X_t)-\mathsf{o}(G, \mathcal{A})\right) : \mathcal{A}\text{ is a 
}T \text{-partition}  
\right\rbrace.  \]
\end{thm}

Let us define $ E(\mathcal{A}):= \bigcup_{t\in T}\delta(X_t) $. In Theorem \ref{t: Mader T-path}, for a system $ \mathcal{P} $ 
of edge-disjoint  $ T $-paths and a $ T $-partition $ \mathcal{A} $ we have equality if and only if the following conditions hold:

\begin{cond}[complementary slackness]\label{cd: Mader T-pathWeak}\
\begin{enumerate}
\item Each $ P\in \mathcal{P} $ uses either  only a single edge from $ E(\mathcal{A}) $ (which must connect two vertex 
sets in $ \mathcal{A} $) or two edges incident with  a component of  $ G-\bigcup \mathcal{A} $.
\item For each component $ Y $ of $ G-\bigcup \mathcal{A} $,  the path-system $ \mathcal{P} $ uses all but at most one edge  
from 
$ \delta(Y) $.
\end{enumerate}
\end{cond} 

\begin{conj}\label{conj: complementarity Weak}
Let $G$ be a (possibly infinite) graph and let $T\subseteq V(G)$. Then there exists a system 
$\mathcal{P}$ of edge-disjoint $T$-paths such that there is a $ T $-partition $ \mathcal{A} $ satisfying Condition 
\ref{cd: Mader T-pathWeak}. 
\end{conj}

Although the Lovász-Cherkassky theorem \ref{t:LCh} is a special case of Mader's edge-disjoint $ T $-path theorem 
\ref{t: Mader T-path},  Conjecture \ref{conj: complementarity Weak} does not seem to imply Conjecture \ref{conj: LCh}. 
This (together with the behaviour of $ T $-joins, see \cite[Theorem 2]{Joo17Tjoins}) 
motivates to formulate a stronger conjecture based on the extension of the concept of 
obstructive components.

 For a possibly infinite graph $ G $, we define a component $ Y $ of  $ G-\bigcup 
\mathcal{A} $ to be \emph{obstructive} if after 
the contraction of $ V(G)\setminus Y $ to a vertex $ v $ the resulting graph $ H $ does not contain a set of pairwise 
edge-disjoint 
cycles 
covering $ \delta_H(v) $. This extends our previous definition of obstructive. Indeed, on the one 
hand, if $ 
d_G(Y) $ is odd, then  $ d_{H}(v) 
$ is the same odd number and hence $ \delta_H(v) $ cannot be covered by edge-disjoint cycles. On the 
other hand, if $ d(Y) $ is even, then finding the desired cycles is equivalent to finding a $ J $-join in the connected graph $ G[Y] $ 
where $ J $ consists of those $ u\in Y $  for which there are odd number of edges between $ u $ and $ v $ in $ H $.

\begin{cond}\label{cd: Mader T-path}\
\begin{enumerate}
\item Each $ P\in \mathcal{P} $ uses either  only a single edge from $ E(\mathcal{A}) $ (which must connect two vertex 
sets in $ \mathcal{A} $) or two edges incident with  a component of  $ G-\bigcup \mathcal{A} $.

\item The path-system $ \mathcal{P} $ uses all the edges $ E(\mathcal{A}) $ except one from $ \delta(Y) $ for each 
obstructive component $ 
Y $.
\end{enumerate}
\end{cond}

Note that if $ G $ is inner Eulerian, then there cannot be any obstructive components (regardless of the choice of $ \mathcal{A} 
$) and 
therefore 
 by replacing Condition \ref{cd: Mader T-pathWeak} with Condition \ref{cd: Mader T-path}  in  Conjecture 
 \ref{conj: complementarity Weak} it will imply Conjecture \ref{conj: LCh}.  We also point out that for finite graphs Conditions 
 \ref{cd: Mader T-pathWeak} and \ref{cd: Mader T-path} are 
equivalent because if $ d(Y) $ is even, 
then $ \mathcal{P} $ cannot miss exactly one edge from $ \delta(Y) $.

Recall that a system $ \mathcal{P} $ of edge-disjoint/vertex-disjoint/internally vertex-disjoint $ T $-paths is called 
\emph{strongly maximal} 
if $ \left|\mathcal{Q}\setminus \mathcal{P}\right|\leq \left|\mathcal{P}\setminus \mathcal{Q}\right| $   for every 
edge-disjoint/vertex-disjoint/internally vertex-disjoint  system $ \mathcal{Q} $ of $ 
T $-paths. 

\begin{conj}\label{conj: edgeMaderEquiv}
Let $G$ be a (possibly infinite) graph and let $T\subseteq V(G)$. Then for  a system $\mathcal{P}$ of edge-disjoint 
$T$-paths the following statements are equivalent:
\begin{enumerate}[label=(\roman*)]
\item $ \mathcal{P} $ is a strongly maximal system of edge-disjoint $ T $-paths.
\item  There exists a $ T $-partition $ \mathcal{A} $ satisfying 
Condition \ref{cd: Mader T-pathWeak} with $ \mathcal{P} $.
\item There exists a $ T $-partition $ \mathcal{A} $ satisfying 
Condition \ref{cd: Mader T-path} with $ \mathcal{P} $.
\end{enumerate}
\end{conj}
Notice  that  $ (iii)\Longrightarrow (ii) \Longrightarrow (i) $.   Indeed, the implication $ (iii)\Longrightarrow 
(ii)$ 
is trivial. Assuming  $ (ii)$,  $ \mathcal{P} $ must 
be an inclusion-wise maximal system of edge-disjoint $ T $-paths. If $ \left| \mathcal{P}\setminus \mathcal{Q}\right|=\kappa 
\geq \aleph_0 $, then $ \left| E(\mathcal{P}\setminus \mathcal{Q})\right|=\kappa $ and since each $ P\in \mathcal{Q}\setminus 
\mathcal{P} $ must contain an edge from $ E(\mathcal{P}\setminus \mathcal{Q}) $, we obtain $ \left|\mathcal{Q}\setminus 
\mathcal{P}\right|\leq \kappa $. If  $ \left| \mathcal{P}\setminus \mathcal{Q}\right|= k\in \mathbb{N} $, then let $ 
G':=G-E(\mathcal{P}\cap \mathcal{Q}) $. Then $ d_{G'}(Y) $ is finite for every component of $ G'-\bigcup \mathcal{A} $ and 
for all of but finitely many $ Y $ it is $ 0 $, moreover,

\[ \frac{1}{2}\left( \sum_{t\in T}d_{G'}(X_t)-\mathsf{o}(G', \mathcal{A})\right)=k, \]

from which $ \left|\mathcal{Q}\setminus \mathcal{P}\right|\leq k $ follows. Thus $ \mathcal{P} $ is strongly maximal. Hence
for establishing Conjecture \ref{conj: edgeMaderEquiv} it is sufficient to prove $ (i)\Longrightarrow (iii) $.

\subsection{Vertex-disjoint T-paths}
 If $ 
 T=V(G) $, then a vertex-disjoint 
 system of $ 
T $-paths is a matching. Infinite matching theory was intensively investigated and is well-understood (see the survey 
\cite{aharoni1991infinite}). The existence of a strongly maximal matching first in countable and then in arbitrary graphs 
was proven by Aharoni (see 
\cite{aharoni1984generalization, aharoni1988matchings}) together with the following theorem:

\begin{thm}[Aharoni,  {\cite[Theorem 5.2]{aharoni1991infinite}}]\label{t: stmMatching}
In every (possibly infinite)  graph $ G=(V,E) $ there is a matching $ M\subseteq E $ such that there is an $ X\subseteq V $ 
with the following properties:
\begin{enumerate}
\item  For each component $ Y $ of $ G-X $, the edges in $ M $ spanned by $ Y $  cover all but at most one 
 vertex of $ Y $. 
\item The vertices in $ X $ are covered by $ M $ in such a way that $ X $ does not span any edge in $ M $.
\item $ G[Y] $ is factor-critical\footnote{A graph is factor-critical if it does not admit a perfect matching but after deleting any 
vertex the resulting graph does.} 
whenever 
$ Y $ is a component of $ G-X $ for which $ M $ does not contain a perfect matching of $ G[Y] $.  
\item Let $ \Pi(G,X) $ be the bipartite graph  whose vertex classes are $ X $ and the set $ \mathcal{Y} $ of the factor-critical 
components of $ 
G-X $, furthermore,  $ xY $ is an edge if $ x $ has a neighbour in $ Y $ in $ G $. Then  for every $ Y\in \mathcal{Y} $ there is 
a 
matching in $ 
\Pi(G,X) $  
covering $ X $ while avoiding vertex $ Y $.
\end{enumerate}
\end{thm}

\begin{rem}
\begin{itemize}\
\item Properties (1) and (2) in  Theorem  \ref{t: stmMatching} are already sufficient to ensure the strong maximality of the 
matching $ M $.
\item For every strongly maximal matching $ M $ there is an $ X $ satisfying (1)-(4).
\item Property (4)  was originally not mentioned by Aharoni but it can be obtained easily by 
applying for example  \cite[Lemma 3.6]{aharoni1983general}.
\item If there is a matching $ M $ for which $ V(M) $ is $ \subseteq $-maximal 
(which is always the case in countable graphs), then the set $ X $ in Theorem  \ref{t: stmMatching} is unique.
\end{itemize} 
\end{rem}

By omitting the assumption of $ T=V(G) $ we leave  matching theory and formulate an infinite generalisation of Gallai's 
theorem \cite{gallai1964maximum}:
\begin{conj}\label{conj: Gallai vertex-disj}
Let $G=(V, E)$ be a (possibly infinite)  graph and let $T\subseteq V$. Then there exists a system 
$\mathcal{P}$ of vertex-disjoint $T$-paths such that there is an $ X\subseteq V $ with the following properties:
\begin{enumerate}
\item For each component $ Y $ of $ G-X $, the paths $\{ P\in \mathcal{P}: V(P)\subseteq Y\} $  cover all but at most one 
 vertex of $ 
T\cap Y $.
\item $ X\subseteq V(\mathcal{P}) $ where $ \left|V(P)\cap X\right|\leq1$ for every $ P\in \mathcal{P} $.
\end{enumerate}
\end{conj}

 A minimax formula for the maximal number of  internally vertex-disjoint $T$-paths was given by Mader in 
\cite{mader1978maximalzahlH}.  We expect the following generalisation based on the complementary slackness conditions 
to be true:

\begin{conj}\label{conj: Mader internally-disj}
Let $G=(V, E)$ be a (possibly infinite)  graph and let $T\subseteq V$. Then there 
exists a 
system 
$\mathcal{P}$ of internally vertex-disjoint $T$-paths such that there is an $ X\subseteq V\setminus T $ and a partition $ 
\mathcal{Y} $ of 
$ V\setminus (T\cup X) $ with the following properties:

\begin{enumerate}
\item[(0)] After the deletion of the vertex set $ X $ and the edges of the subgraphs $ G[Y] $ for $ Y\in \mathcal{Y} $   the 
resulting graph does not contain any $ T $-path. 
\end{enumerate}   \[ \text{Let }B_Y:= \{ v\in Y: v \text{ has a 
neighbour in 
} V\setminus (X\cup Y) 
\}\text{ for }Y\in \mathcal{Y}. \] 
\begin{enumerate}
\item The paths in $ \mathcal{P} $ cover $ X $ and all but at most one vertex of $ B_Y $ for every $ Y\in \mathcal{Y} $.
\item For every $ P\in \mathcal{P} $ either $ \left|V(P)\cap X\right|=1 $ and $ \left|V(P)\cap B_Y\right|\leq 1$ for every $ 
Y\in 
\mathcal{Y} $ or $ \left|V(P)\cap X\right|=0 $ and there is a unique $ Y_P\in \mathcal{Y} $ with $ \left|V(P)\cap 
B_{Y_P}\right|=2 $ 
while $ \left|V(P)\cap B_Y\right|\leq 1$ for $ Y\in \mathcal{Y}\setminus \{ Y_P \} $.
\item For every $ Y\in \mathcal{Y} $ there is at most one $ P\in \mathcal{P} $ with $ \left|V(P)\cap B_Y\right|=1 $.
\end{enumerate}
\end{conj}

\begin{conj}
The systems of $ T $-paths described in Conjecture \ref{conj: Gallai vertex-disj} (Conjecture \ref{conj: Mader 
internally-disj}) are exactly the strongly maximal systems of  vertex-disjoint (internally vertex-disjoint)  $ T $-paths.
\end{conj}
\printbibliography
\end{document}